\documentclass[12 pt]{amsart}
\title{SOBRE ECUACIONES DIFERENCIALES Y CALCULO DE VARIACIONES EN LOS TRABAJOS DE L. EULER}
\author{Jonathan Taborda}
\email{taborda50@gmail.com}
\date{Mayo 22 de 2007}
\usepackage{amsmath,amssymb,amsthm,pb-diagram,mathpazo,amsrefs}
\usepackage[activeacute,spanish]{babel}
\usepackage[T1]{fontenc}
\newtheorem{lem}{Lema}
\dedicatory{A Leonhard Euler, con motivo de su tricentenario 1707-2007}
\begin{document}\maketitle
\begin{abstract}
En las siguientes notas, se pretende realizar una pequeña digresión, relativa a los tópicos mencionados en el título de la misma; puesto que tales no fueron abordados en el pasado Homenaje realizado por el Instituto de Física, de la U de A. Consideramos más que perogrullo tratar de justificar la importancia y vigencia en el desarrollo de las Matemáticas y Física a lo largo del Siglo XVIII como en la actualidad de tales temas; por ende presentamos una breve descripción de los métodos y problemas atacados por Euler y sus contemporáneos empleando dicha Heurística.\\
De antemano señalamos que éstas, constituyen un intento vehementemente pauperrimo, para honrar la memoria de quien es considerado el Shakespeare de las Matemáticas: $\textit{Universal, rico en detalles e inagotable}.$\\\\
\end{abstract}
\fontfamily{ppl}\selectfont
\section{Las Ecuaciones Diferenciales Ordinarias en el Siglo XVIII.}
\subsection{EDO de primer orden.}
Los primeros trabajos en ED, como ocurrió con el cálculo infinitesimal a finales del Siglo XVII y comienzos del Siglo XVIII, vieron primero la luz en cartas de un matem\'{a}tico a otro, muchas de las cuales ya no est\'{a}n disponibles, o en publicaciones que a menudo repiten los resultados establecidos o reivindicados en cartas.\\\\
 En las $\textit{Actas Eruditorum}$ de 1693$\footnotemark\footnotetext{{\OE}uvres, 10, 512-514.}$ Huygens explícitamente habla de ED, y Leibnitz, en otro artículo de la misma revista y año$\footnotemark\footnotetext{Math.Schriften, 5, 306.}$ dice que las ED son funciones de elementos de triángulo característico.\\\\
Lo primero que normalmente aprendemos de las ED que aparecen al eliminar las constantes arbitrarias entre una función dada y sus derivadas, no se hizo hasta 1740, aproximadamente y se debe a Alexis Fontaine del Bertins.\\\\
Jacques Bernoulli fue de los primeros en utilizar el cálculo infinitesimal para resolver problemas de EDO; en mayo de 1690$\footnotemark\footnotetext{Acta Erud., 1690, 217-219=Opera,1, 421-424.}$ publicó su solución al problema de la Isócrona, si bien Leibnitz ya había dado una solución analítica; este problema consiste en encontrar una curva a lo largo de la cual un péndulo tarde el mismo tiempo en efectuar una oscilación completa, sea grande o pequeño el arco que recorre; la ecuaci\'{o}n, en los símbolos de Bernoulli, era \begin{gather}dy\sqrt{b^2y-a^3}=dx\sqrt{a^3}\end{gather}\\\\
Bernoulli concluía de la igualdad de diferencias que las integrales (la palabra se utilizaba por vez primera) han de ser iguales y dio como solución \begin{gather}\dfrac{2b^2y-2a^3}{3b^2}\sqrt{b^2y-a^3}=x\sqrt{a^3}\end{gather}\\\\
La curva es, por supuesto, la cicloide.\\\\
En el mismo art\'{i}culo de 1690, Jacques Bernoulli planteó el problema de encontrar la curva que adopta una cuerda flexible e inextensible colgada libremente de dos puntos fijos, la curva que Leibnitz denominó catenaria.\\\\
En las $\textit{Acta}$ de junio de 1691, Leibnitz, Huygens y J. Bernoulli publicaron soluciones independientes.\\\\
Leibnitz descubrió la técnica de separación de variables y la comunicó en una carta de 1691 a Huygens; resolvió así una ecuación de la forma $y(dx/dy)=f(x)/g(y),$ escribiendo $dx/f(x)=g(y)/d(y)$ y consiguió integrar ambos miembros; no formuló el método general.\\\\
En 1694, Leibnitz y J. Bernoulli introdujeron el problema de encontrar la curva o familia de curvas que cortan con un ángulo dado a una familia de curvas dadas; J. Bernoulli trayectorias a las curvas secantes.\\\\
También fueron identificadas las ED de primer orden exactas, i.e, las ecuaciones $M(x,y)dx+N(x,y)dy=0,$ para las cuales $Ndx+Ndy$ es la diferencia exacta de una función $z=f(x,y).$ Clairaut, célebre por su trabajo sobre la forma de la Tierra, había dado la condición $\partial M/\partial y=\partial N/\partial x$ para que la ecuación fuese exacta en sus artículos de 1739 y 1740, condición que también fue dada independientemente por Euler en un artículo escrito en 1734-35$\footnotemark\footnotetext{Comm. Acad. Sci. Petrop., 7, 1734/35, 174-193, pub. 1740=Opera, (1), 22, 36-56.}$\\\\
\subsection{Soluciones Singulares.}
Las soluciones singulares no se obtienen de la solución  general dando un valor concreto a la constante de integración; i.e, no son soluciones particulares.\\\\
Clairaut y Euler habían desarrollado un método para hallar la solución singular a partir de la propia ecuación, a saber, eliminando $y'$ de $f(x,y,y')=0$ y $\partial f/\partial y'=0.$\\\\
Este hecho y el de que las soluciones singulares no estén contenidas en la solución general intrigaron a Euler; en sus $\textit{Institutions}$ de 1768$\footnotemark\footnotetext{Volumen I, p\'{a}gs. 393 y ss.}$ dio un criterio para distinguir cuando no se conocía la solución general, criterio que fue mejorado por D'Alembert$\footnotemark\footnotetext{Hist. de l'Acad. des. Sci. Paris, 1768, 85 y ss., pub. 1772.}$.\\\\
\section{Ecuaciones de segundo orden.}
Las ED de segundo orden aparecen en problemas físicos ya en 1691. J. Bernoulli se planteó el problema de la forma de una vela bajo la presión del viento, el problema de la velaria, lo que le llevó a la ecuación de segundo orden $d^2x/ds^2=(dy/ds)^3,$ donde s es la longitud de arco.\\\\
Tales ecuaciones aparecen a continuaci\'{o}n al tocar el problema de determinar el perfil de una cuerda el\'{a}stica que vibra, sujeta por los extremos-v.g una cuerda de violin.\\\\
Euler comenz\'{o} a considerar ED de segundo orden en 1728. Su inter\'{e}s en ellas fue suscitado en parte por sus trabajos en mec\'{a}nica; $\footnotemark\footnotetext{Cr. C.H.Truesdell. \textit{Ensayos sobre Historia de la Mec\'{a}nica}}$ hab\'{i}a trabajado, V.g, sobre el movimiento del p\'{e}ndulo en medios con rozamiento, lo que conduce a Ec. de segundo orden; trabaj\'{o} para el Rey de Prusia sobre el efecto de la resistencia del aire sobre los proyectiles; en esto tom\'{o} el trabajo del Ingl\'{e}s Benjamin Robins, lo mejor\'{o} y escribi\'{o} una versi\'{o}n en alem\'{a}n (1745), la cual fu\'{e} traducida al franc\'{e}s y al ingl\'{e}s e utilizada en artiller\'{i}a.\\\\
Consider\'{o} tambi\'{e}n una clase de Ec. de segundo orden que redujo mediante un cambio de variables a Ec. de primer orden. Consider\'{o}, V.g., la ecuaci\'{o}n \begin{gather}ax^mdx^p=y^pdy^{p-2}d^2y\end{gather} \'{o} en la forma de derivadas, \begin{gather}\Biggl(\dfrac{dy}{dx}\Biggr)^{p-2}\dfrac{d^2y}{dx^2}=\dfrac{ax^m}{y^n}\end{gather}\\\\
Euler introdujo las nuevas variables $t$ y $v$ por medio de las ecuaciones \begin{gather}y=e^vt(v),\hspace{0.5cm} x=e^{\alpha v}\end{gather}\\\\
donde $\alpha$ es una constante a determinar.\\\\
Las ecuaciones (5) se pueden contemplar como las Ec. param\'{e}tricas de $x$ e $y$ en t\'{e}rminos de $v,$ de modo que se pueden calcular $dx/dy$ y $d^2/dx^2$ y, sustituyendo en (4), obtener una Ec. de segundo orden en $t$ como funci\'{o}n de $v.$ Euler fija entonces $\alpha$ de modo que quede eliminado el factor exponencial y $v$ ya no aparezca expl\'{i}citamente; una nueva transformaci\'{o}n, a saber, $z=dv/dt,$ reduce la Ec. de segundo orden a una de primero.\\\\
  Despu\'{e}s de haber abordado el tema de los sonidos musicales en el libro\\ $\textit{Tentamen Novae Theoriae Musicae ex Certissimis Harmoniae Principiis Dilucide Expositae}$ (Investigaci\'{o}n sobre una nueva teor\'{i}a de la m\'{u}sica, claramente expuesta a partir de incontestables principios de la armon\'{i}a), escrito antes de 1731 y publicado en 1739 $\footnotemark\footnotetext{Opera, (3),1,197-427}$, Euler prosigui\'{o} el trabajo de Daniel Bernoulli, salvo que los argumentos matem\'{a}ticos de \'{e}ste son m\'{a}s claros. Para una forma de cadena continua, a saber, el caso especial en que el peso es proporcional a $x^n,$ Euler tuvo que resolver \begin{gather}\dfrac{x}{n+1}\frac{d^2y}{dx^2}+\dfrac{dy}{dx}+\dfrac{y}{\alpha}=0\end{gather}\\\\
  Obtiene la soluci\'{o}n en serie que en notaci\'{o}n moderna est\'{a} dada por $\footnotemark\footnotetext{Para v arbitrario (incluyendo valores complejos):\\
  $I_v(z)=\sum\limits_{n=0}^{\infty}\dfrac{(z/2)^{v+2n}}{n!\Gamma(v+n+1)}$.\\
  Los $I_v(z)$ reciben el nombre de funciones Bessel modificadas.}$ \begin{gather}y=Aq^{-n/2}I_n(2\sqrt{q}),\hspace{0.5cm} q=-\dfrac{(n+1)x}{\alpha}\end{gather}\\\\
La $n$ es aqu\'{i} general, por lo que Euler est\'{a} introduciendo funciones de Bessel de \'{i}ndice real arbitrario. Da tambi\'{e}n la soluci\'{o}n en forma integral \begin{gather}y=A\dfrac{\int_0^1{(1-t^2)}^{(2n-1)/2}\cosh\Biggl(2t\sqrt{\dfrac{(n+1)x}{\alpha}}\Biggr)dt}{\int_0^1{(1-\tau^2)}^{(2n-1)/2}d\tau}\end{gather}\\\\
Este es quiz\'{a} el primer caso de soluci\'{o}n de una ED de segundo orden expresada como integral.\\\\
En un art\'{i}culo de 1739$\footnotemark\footnotetext{Comm. Acad. Sci. Petrop., 1739, 128-149, publi. 1750=Opera, (2), 10, 78-97}$ Euler se ocup\'{o} de las ED del oscilador arm\'{o}nico, $\ddot{x}+Kx=0$, y de las oscilaciones forzadas del mismo \begin{gather}M\ddot{x}+Kx=F\sin w_\alpha t\end{gather}\\\\
obtuvo las soluciones por cuadraturas y descubri\'{o} (redescubri\'{o}, en realidad, ya que otros lo hab\'{i}an encontrado antes) el fen\'{o}meno de resonancia; a saber, que si $w$ es la frecuencia natural $\sqrt{K/M}$ del oscilador, que se obtiene cuando $F=0,$ entonces, cuando $w_\alpha/w$ se aproxima a 1, la oscilaci\'{o}n forzada tiene amplitud cada vez m\'{a}s grande y se hace infinita.\\\\
En 1760, Euler $\footnotemark\footnotetext{Novi Comm. Acad. Sci. Petrop., 9, 1760/61, publi. 1763=Opera, (1), 22, 334-394, y 9, 1762/63, 154-159, publi. 1764=Opera, (1), 22, 403-420}$ consider\'{o} la ecuaci\'{o}n de Riccati \begin{gather}dz/dx+z^2=ax^n\end{gather}\\
y demostr\'{o} que si se conoce una integral particular $v,$ entonces la transformaci\'{o}n \begin{gather}z=v+u^{-1}\end{gather}\\
convierte a aquella en una ecuaci\'{o}n lineal. Adem\'{a}s, si se conocen dos integrales particulares, la integraci\'{o}n de la ecuaci\'{o}n original se reduce a cuadraturas.\\\\
\subsection{Ecuaciones de Orden Superior.}
En diciembre de 1734, Daniel Bernoulli escribi\'{o} a Euler, quien estaba en San Petersburgo, que hab\'{i}a resuelto el problema del desplazamiento transversal de una barra el\'{a}stica (un cuerpo unidimensional de madera o de acero) fijada a una pared en uno de sus extremos y libre el otro. Bernoulli hab\'{i}a obtenido la ED \begin{gather}K^4\dfrac{d^4y}{dx^4}=y\end{gather}\\
donde $K$ es una constante, $x$ es la distancia desde el extremo libre de la barra e $y$ el desplazamiento vertical en ese punto respecto a la posici\'{o}n sin pandeo de la barra.\\\\
Euler, en una r\'{e}plica escrita antes de junio de 1735, afirm\'{o} que el tamb\'{e}n descubierto esta ecuaci\'{o}n y que no era capaz de integrarla salvo utlizando series, y que hab\'{i}a obtenido cuatro series distintas; estas series representaban funciones circulares y exponenciales; pero Euler no lo v\'{i}o entonces.\\\\
  Cuatro años m\'{a}s tarde, en una carta a J. Bernoulli (15 de Septiembre de 1739), Euler indicaba que su soluci\'{o}n se pod\'{i}a representar como \begin{gather}y=A\Biggl[\bigl(\cos x/K + \cosh x/K\bigr)-\dfrac{1}{b}\bigl(\sin x/K
+ \sinh x/K\bigr)\Biggr]\end{gather}\\
donde $b$ est\'{a} determinado por la condici\'{o}n $y=0$ cuando $x=l,$ de modo que \begin{gather}b=\dfrac{\sin l/K + \sinh l/K}{\cos l/K + \cosh l/K}\end{gather}\\
Los problemas de elasticidad condujeron a Euler a considerar el problema matem\'{a}tico de la resoluci\'{o}n de Ec. lineales generales con coeficientes constantes, y en una carta a Jean Bernoulli de septiembre de 1739 escribe que hab\'{i}a tenido \'{e}xito.\\\\
Bernoulli le respondi\'{o} afirmando que \'{e}l ya hab\'{i}a considerado tales ecuaciones en 1700 incluso con coeficientes variables.\\\\
En la publicaci\'{o}n de su trabajo $\footnotemark\footnotetext{Misc. Berolin., 7, 1743, 193-242=Opera, (1), 22, 108-149}$, Euler consider\'{o} la ecuaci\'{o}n \begin{gather}0=Ay+B\dfrac{dy}{dx}+C\dfrac{d^2y}{dx^2}+D\dfrac{d^3y}{dx^3}+\ldots+L\dfrac{d^ny}{dx^n}\end{gather}\\
donde los coeficientes son constantes; la ecuaci\'{o}n se dice que es homogenea porque el t\'{e}rmino independiente de $y$ y sus derivadas es 0. Euler indica que la soluci\'{o}n general ha de contener $n$ constantes arbitrarias y que su soluci\'{o}n vendr\'{a} dada por la suma de $n$ soluciones particulares, cada una de ellas multiplicada por una constante; hace entonces la sustituci\'{o}n \begin{gather}y=\exp\bigl[\int r dx\bigr]\end{gather}\\
con $r$ constante, y obtiene la ecuaci\'{o}n en $r$ \begin{gather}A+Br+Cr^2+\ldots+Lr^n=0\end{gather}\\
que se denomina ecuaci\'{o}n caracter\'{i}stica o auxiliar. Cuando $q$ es una ra\'{i}z real simple de esta ecuaci\'{o}n, entonces \begin{gather}ae^{\int qdx}\end{gather}\\
es una soluci\'{o}n de la ED original. Si la ecuaci\'{o}n caracter\'{i}stica tiene una ra\'{i}z m\'{u}ltiple $q$, Euler hace $y=e^{qx}u(x)$ y sustituye en la ED; obteniendo que \begin{gather}y=e^{qx}\bigl(\alpha+\beta x+\gamma x^2+\ldots\ldots+\varsigma x^{k-1}\bigr)\end{gather}\\
es una soluci\'{o}n que contiene $k$ constantes arbitrarias si $q$ aparece $k$ veces como ra\'{i}z de la ecuaci\'{o}n caracter\'{i}stica. Trat\'{o} tambi\'{e}n los casos de ra\'{i}ces complejas conjugadas y de ra\'{i}ces complejas m\'{u}ltiples, con lo que Euler resuelve completamente las Ec. lineales homogeneas con coeficientes constantes.\\\\
  Algo m\'{a}s tarde $\footnotemark\footnotetext{Novi Comm. Acad. Sci. Petrop., 3, 1750/51, 3-35, publi. 1753=Opera(1), 22, 181-213}$ estudi\'{o} la EDO lineal de orden $n$ no homogenea; su m\'{e}todo consisti\'{o} en multiplicar la Ec. por $e^{\alpha x}$, integrar ambos miembros y proceder a determinar $\alpha$ de modo que la Ec. se reduzca a una de orden inferior. As\'{i}, V.g, para resolver \begin{gather}C\dfrac{d^2y}{dx^2}+B\dfrac{dy}{dx}+Ay=X(x),\end{gather}\\
  multiplicada por $e^{\alpha x}dx$ y obtiene \begin{gather}\int\Biggl[e^{\alpha x}C\dfrac{d^2y}{dx^2}+e^{\alpha x}B\dfrac{dy}{dx}+e^{\alpha x}Ax\Biggr]dx=\int e^{\alpha x}X dx\end{gather}\\
Pero, para $A',b'$ y $\alpha$ apropiados, el primer miembro ha de ser \begin{gather}e^{\alpha x}\bigl(A'y+B'dy/dx\bigr)\end{gather}\\
Derivando esta expresi\'{o}n y comparando con la ecuaci\'{o}n original Euler obtiene que \begin{gather}B'=C, A'=B-\alpha C, A'=A/\alpha\end{gather}\\
con lo cual, de las \'{u}ltimas ecuaciones, \begin{gather}A-B\alpha+C\alpha^2=0\end{gather}\\
Quedando, pues, determinados $\alpha, A'$ y $B'$ y la ecuaci\'{o}n original se reduce a \begin{gather}A'y+B'dy/dx=e^{-\alpha x}\int e^{\alpha x}X dx\end{gather}\\
Un factor integrante de esta ecuaci\'{o}n es $e^{\beta x}dx,$ donde $\beta=A'/B',$ de modo que por (23), se tiene $\alpha\beta=A/C$ y $\alpha+\beta=B/C$ y consecuentemente, por (24), $\alpha$ y $\beta$ son las ra\'{i}ces de $A-B\alpha+C\alpha^2=0.$\\\\
\section{el c\'{a}lculo de variaciones en siglo XVIII.}
\subsection{La teor\'{i}a de Superficies}
Como la teor\'{i}a de curvas en el espacio, la teor\'{i}a de superficies tuvo un inicio lento. Empez\'{o} con el estudio de las geod\'{e}sicas sobre superficies, con las geod\'{e}sicas sobre la Tierra como precauci\'{o}n principal. En el $\textit{Journal des Scavans}$ de 1697, J. Bernoulli propuso el problema de encontrar m\'{i}nimo entre dos puntos sobre una superficie convexa $\footnotemark\footnotetext{Opera, 1, 204-05.}$ Le escribi\'{o} a Leibnitz en 1698 para señalarle que el plano osculador (el plano del c\'{i}rculo osculador) en cualquier punto de una geod\'{e}sica es penperdicular a la superficie en ese punto. En 1698 Jacques Bernoulli resolvi\'{o} el problema de las geod\'{e}sicas sobre cilindros, conos y superficies de revoluci\'{o}n. El m\'{e}todo era limitado, a pesar de que en 1728 Jean Bernoulli $\footnotemark\footnotetext{Opera, 4, 108-128.}$ tuvo \'{e}xito con el m\'{e}todo y encontr\'{o} las geod\'{e}sicas sobre otros tipos de superficies.\\\\
En 1728 Euler $\footnotemark\footnotetext{Comon. Acad. Sci. Petrop. 3, 1728, 110-124, pub. 1732=Opera(1), 25, 1-12.}$ proporcion\'{o} las ED para las geod\'{e}sicas sobre superficies. Euler us\'{o} el m\'{e}todo que hab\'{i}a introducido en el c\'{a}lculo de variaciones. En 1760, en su \\
$\textit{Recherches sur courbore des surfaces}$ (Investigaciones sobre la curvatura de las superficies) $\footnotemark\footnotetext{M\'{e}m. de. l'Acad. de Berlin, 16, 1760, 119-143, publ. 1767=Opera(1), 28. 1-22.}$, Euler estableci\'{o} la teor\'{i}a de superficies. Este trabajo es la contribuci\'{o}n m\'{a}s importante de Euler a la geometr\'{i}a diferencial y un punto culminante de la materia. Euler representa una superficie por $z=f(x,y)$ e introduce la actual notaci\'{o}n \begin{gather}p=\dfrac{\partial z}{\partial x}, q=\dfrac{\partial z}{\partial y}, r=\dfrac{\partial^2 z}{\partial x^2}, s=\dfrac{\partial^2 z}{\partial x\partial y}, t=\dfrac{\partial^2 z}{\partial x^2}\end{gather}\\\\
M\'{a}s adelante dice:
\begin{quote}
 $\textit{Empiezo por determinar el radio de curvatura de cualquier secci\'{o}n  plana}$\\
 $\textit{de una superficie; entonces aplico esta soluci\'{o}n a secciones que son perpendiculares a la }$\\ $\textit{superficie en cualquier punto dado, comparo los radios de curvatura de estas secciones }$\\ $\textit{con respecto a su mutua inclinaci\'{o}n, lo que nos coloca en situaci\'{o}n de establecer una}$\\$\textit{idea adecauda de la curvatura de superficies}.$
 \end{quote}
\subsection{Los primeros trabajos de Euler.}
En 1728, Jean Bernoulli propuso a Euler el problema de obtener geod\'{e}sicas sobre superficies aplicando la propiedad de que los planos osculadores de las geod\'{e}sicas cortan la superficie en \'{a}ngulos rectos. Este problema inici\'{o} a Euler en el c\'{a}lculo de variaciones. Lo resolvi\'{o} en 1728 $\footnotemark\footnotetext{\'{I}bid. nota 17.}$ En 1734 Euler generaliz\'{o} el problema de la braquist\'{o}crona para minimizar cantidades diferentes de tiempo, y tomando en cuenta un medio resistente $\footnotemark\footnotetext{Comm. Acad. Sci. Petrop, 7, 1734/35, 135-149. Opera(1), 25, 41-53.}$\\\\
M\'{a}s adelante, Euler se propuso encontrar una aproximaci\'{o}n m\'{a}s general a problemas en este terreno. Su m\'{e}todo, que fue una simplificaci\'{o}n del de Jacques Bernoulli; consisti\'{o} en reemplazar la integral de un problema por una suma y reemplazar las derivadas en el integrando por coeficientes diferenciales, haciendo la integral una funci\'{o}n de n\'{u}mero finito de ordenadas del arco $y(x).$ M\'{a}s adelante, vari\'{o} uno o m\'{a}s de las ordenadas seleccionadas arbitrariamente y calcul\'{o} la variaci\'{o}n en la integral. Igualando la variaci\'{o}n de la integral a cero y usando un proceso de paso al l\'{i}mite muy tosco para transformar las ED resultantes, obtuvo la ED que deb\'{i}a ser satisfecha por el arco minimizante.\\\\
Por el m\'{e}todo descrito con anterioridad, aplicado a integrales de la forma \begin{gather}J=\int_{x_1}^{x_2}f(x,y,y')dx\end{gather}\\ Euler tuvo \'{e}xito al demostrar que la funci\'{o}n $y(x)$ que minimiza o maximiza el valor de $J$ debe satisfacer la EDO \begin{gather}f_y-\dfrac{dy}{dx}(f_{y'})=0\end{gather}\\\\
Consideremos ahora una curva regular $\mathcal{G}$ sobre una superficie y que pasa por dos puntos A y B. Sea $F(t,x^1,x^2,x^3,\dot{x}^1,\dot{x}^2,\dot{x}^3)$ una funci\'{o}n con derivadas parciales en un cierto abierto de $\mathbb{R}^7$. Consideremos  \[\int_{t_0}^{t_1}F(t,x^i,\dot{x}^i)dt\]\\\\
Podemos pensar en la siguiente funci\'{o}n: \[\begin{diagram}\node{J:\mathcal{G}}\arrow{e,t}\\
\node{\mathbb{R}}
\end{diagram}\]
\[\begin{diagram}
\node[2]{x^i(t)}\arrow{e,t}\\
\node{J\bigl(x^i(t)\bigr)=\int_{t_0}^{t_1}F(t,x^i,\dot{x}^i)dt}\end{diagram}\]\\
$\mathcal{G}$ es el cojunto de curvas regulares que unen A con B. La funci\'{o}n $J$ se denomina un funcional.\\\\
Nos interesan las curvas $x^i(t)$ que hagan m\'{i}nima a  $J$. Este problema es un problema del c\'{a}lculo Variacional.\\\\
Sea $x^i(t)$ una curva que hace m\'{a}x(m\'{i}nima) a \[\int_{t_o}^{t_1}F(t,x^i,\dot{x}^i)dt\] Esa curva se llamar\'{a} una curva extremal del funcional $J.$\\\\
Sea
Veamos que la curva de Ecuaciones param\'{e}tricas $\xi^i(t)$ pasa por A y B.\\
$$\xi^i(t_0)=x^i(t_0)+\varepsilon\mathcal{H}^i(t_0)=T^{-1}(a^i)\hspace{0.5cm}\therefore T\bigl(\xi^i(t_0)\bigr)=a^i.$$\\
igualmente se prueba para $b^i.$\\\\
Sea \[J\bigl(\xi^i(t)\bigr)=\int_{t_0}^{t_1}F(t,\xi^i,\dot{\xi}^i)dt=J(\varepsilon)\]\\
\[=\int_{t_o}^{t_1}F(t,x^i+\varepsilon\mathcal{H}^i,\dot{x}^i+\varepsilon\dot{\mathcal{H}^i})dt\]\\
\[J(0)=J\bigl(x^i(t)\bigr)=\int_{t_0}^{t_1}F(t,x^i,\dot{x}^i)dt\]\\
\[J(\varepsilon)-J(0)=\int_{t_0}^{t_1}F(t,x^i+\varepsilon\mathcal{H}^i,\dot{x}^i+\varepsilon\dot{\mathcal{H}^i}dt)-\int_{t_0}^{t_1}F(t,x^i,\dot{x}^i)dt\]\\
\begin{gather}=\int_{t_0}^{t_1}\bigl\{F(t,x^i+\varepsilon\mathcal{H}^i,\dot{x}^i+\varepsilon\dot{\mathcal{H}^i})-F(t,x^i,\dot{x}^i)\bigr\}dt\end{gather}\\
Ahora $F$ es una funci\'{o}n de $(t,x^1,x^2,x^3,\dot{x}^1,\dot{x}^2,\dot{x}^3)$\\\\
Para estimar la expresi\'{o}n (29) debemos recordar la F. de Taylor de $2^{do}$-orden.\\
Sea $\begin{diagram}\node[2]{f:B_{(a)}\subset\mathbb{R}^n}\arrow{e,t}\\
\node{\mathbb{R}}\end{diagram}$\\
$\begin{diagram}\node[2]{x}\arrow{e,t}\\
\node{f(x)=f(x^1,\ldots,x^n)}\end{diagram}$\\\\
un campo escalar cuyas derivadas parciales mixtas $\dfrac{\partial f}{\partial x^i\partial y^j}, i,j=1,2,\ldots,n$ existen y son continuas en la $B_{(a)}.$\\
Entonces $\forall y\in\mathbb{R}^n$ tal que $a+y\in B_{(a)},$ si consideramos el siguiente conjunto $(a,a+y),\exists\, 0<c<1$ tal que $$f(a+y)=f(a)+\langle\nabla f(a),y\rangle+\frac{1}{2}y^tH(a+cy)y$$\\\\
donde $$H(a+cy)=\begin{pmatrix}\dfrac{\partial^2f}{\partial x^1\partial x^1}&\dfrac{\partial^2f}{\partial x^2\partial x^1}&\ldots\ldots&\dfrac{\partial^2f}{\partial x^n\partial x^1}\\
\vdots&\vdots&&\vdots\\
\dfrac{\partial^2f}{\partial x^1\partial x^n}&\dfrac{\partial^2f}{\partial x^2\partial x^n}&\ldots\ldots&\dfrac{\partial^2f}{\partial x^n\partial x^n}\end{pmatrix}_{a+cy}$$\\\\
En nuestro caso $$F(t_{+0},x^i+\varepsilon\mathcal{H}^i,\dot{x}^i+\varepsilon\dot{\mathcal{H}^i})=$$\\
\[F(t,x^i,\dot{x}^i)+\left\langle\bigl(\dfrac{\partial F}{\partial t},\dfrac{\partial F}{\partial x^i},\dfrac{\partial F}{\partial\dot{x}^i}\bigr),\bigl(0,\varepsilon\mathcal{H}^i,\varepsilon\dot{\mathcal{H}^i}\bigr)\right\rangle+\frac{1}{2}{\bigl(0,\varepsilon\mathcal{H}^i,\varepsilon\dot{\mathcal{H}^i}\bigr)}^tH_c\begin{pmatrix}0\\
\varepsilon\mathcal{H}^i\\
\varepsilon\dot{\mathcal{H}^i}\end{pmatrix}\]\\
$$=F\bigl(t_{+0},x^i+\varepsilon\mathcal{H},\dot{x}^i+\varepsilon\dot{\mathcal{H}^i}\bigr)-F(t,x^i,\dot{x}^i)$$$$=\varepsilon\Biggl(\dfrac{\partial F}{\partial x^i}\mathcal{H}^i+\dfrac{\partial F}{\partial\dot{x}^i}\dot{\mathcal{H}^i}\Biggr)+\underset{\overset{\parallel}{0}}{\underbrace{\varepsilon^2(0,\mathcal{H}^i,\dot{\mathcal{H}^i})H_c\begin{pmatrix}0\\
\mathcal{H}^i\\
\dot{\mathcal{H}^i}\end{pmatrix}}}$$\\\\
As\'{i} las cosas \[J(\varepsilon)-J(0)=\int_{t_0}^{t_1}\varepsilon\Biggl(\dfrac{\partial F }{\partial x^i}\mathcal{H}^i+\dfrac{\partial F}{\partial\dot{x}^i}\dot{\mathcal{H}^i}\Biggr)dt\]\\\\
\[\dfrac{J(\varepsilon)-J(0)}{\varepsilon}=\int_{t_0}^{t_1}\Biggl(\dfrac{\partial F}{\partial x^i}\mathcal{H}^i+\dfrac{\partial F}{\partial\dot{x}^i}\dot{\mathcal{H}^i}\Biggr)dt=\int_{t_0}^{t_1}\dfrac{\partial F}{\partial x^i}\mathcal{H}^idt+\int_{t_0}^{t_1}\dfrac{\partial F}{\partial\dot{x}^i}\dot{\mathcal{H}^i}dt\]\\
\[\int_{t_o}^{t_1}\dfrac{\partial F}{\partial\dot{x}^i}\dot{\mathcal{H}^i}dt=\dfrac{\partial F}{\partial\dot{x}^i}\mathcal{H}^i\Biggr]_{t_o}^{t_1}-\int_{t_0}^{t_1}\Biggl[\dfrac{d}{dt}\Biggl(\dfrac{\partial F}{\partial\dot{x}^i}\Biggr)\Biggr]\mathcal{H}^idt\]\\
\[=\dfrac{\partial F}{\partial\dot{x}^i}\underset{\overset{\parallel}{0}}{\underbrace{\Biggl(\mathcal{H}^i(t_1)-\mathcal{H}^i(t_0)\Biggr)}}-\int_{t_0}^{t_1}\dfrac{d}{dt}\Biggl(\dfrac{\partial F}{\partial\dot{x}^i}\Biggr)\mathcal{H}^idt\]\\
\[=\int_{t_0}^{t_1}\dfrac{\partial F}{\partial x^i}\mathcal{H}^idt-\int_{t_0}^{t_1}\Biggl[\dfrac{d}{dt}\Biggl(\dfrac{\partial F}{\partial\dot{x}^i}\Biggr)\Biggr]\mathcal{H}^idt\]\\\\
Finalmente \[\dfrac{J(\epsilon)-J(0)}{\varepsilon}=\int_{t_0}^{t_1}\Biggl[\dfrac{\partial F}{\partial x^i}-\dfrac{d}{dt}\Biggl(\dfrac{\partial F}{\partial\dot{x}^i}\Biggr)\Biggr]\mathcal{H}^idt\]\\
\[\lim_{\varepsilon\rightarrow 0}\dfrac{J(\varepsilon)-J(0)}{\varepsilon}=\dfrac{dJ}{d\varepsilon}=\int_{t_0}^{t_1}\Biggl[\dfrac{\partial F}{\partial x^i}-\dfrac{d}{dt}\Biggl(\dfrac{\partial F}{\partial\dot{x}^i}\Biggr)\Biggr]\mathcal{H}^idt=0\]\\\\
Luego, el siguiente conjunto de ecuaciones, se conocen como las Ecuaciones de Euler del C\'{a}lculo Variacional \[\begin{cases}\dfrac{\partial F}{\partial x^1}-\dfrac{d}{dt}\Biggl(\dfrac{\partial F}{\partial\dot{x}^1}\Biggr)=0\\\\
\dfrac{\partial F}{\partial x^2}-\dfrac{d}{dt}\Biggl(\dfrac{\partial F}{\partial\dot{x}^2}\Biggr)=0\\\\
\dfrac{\partial F}{\partial x^3}-\dfrac{d}{dt}\Biggl(\dfrac{\partial F}{\partial\dot{x}^3}\Biggr)=0\end{cases}\]\\
\begin{lem}[\textbf{lema fundamental del C\'{a}lculo de Variaciones}]
Sea $\phi(x)$ una funci\'{o}n continua en $[x_0,x_1].$\\
Si $\forall \eta(x);$ diferenciable con derivadas continuas en $[x_0,x_1]$ tal que $\eta(x_0)=\eta(x_1)$ se cumple que \[\int_{x_0}^{x_1}\eta(x)\phi(x)dx=0,\] entonces $\phi=0.$
\end{lem}
\begin{proof}
Razonemos por Reducci\'{o}n al Absurdo.\\\\
Supongamos que $\phi\neq0,$ entonces $\exists\xi\in[x_0,x_1]$ tal que $\phi(\xi)>0.$\\\\
Ahora, como $\phi$ es continua y $\phi(\xi)>0,\hspace{0.2cm}\exists[\xi_0,\xi_1]$ vecindad de $\xi$ tal que $\phi(x)>0,x\in[\xi_0,\xi_1].$\\\\
Definamos ahora \[\begin{diagram}\node{\eta:[x_0,x_1]}\arrow{e,t}\\
\node{\mathbb{R}}\end{diagram}\]
\[\begin{diagram}\node{x}\arrow{e,t}\\
\node{\eta(x)}\end{diagram}\]\\\\
$\begin{cases}=(x-\xi_0)^4(x-\xi_1)^4,&\text{si $x\in[\xi_0,\xi_1]$}\\
=0,\text{en otro caso}
\end{cases}$\\\\
De acuerdo a esto tenemos lo siguiente:\\\\
$\eta(x_0)=0\\
\eta(x_1)=0\\\\
\eta'(x)=4(x-\xi_0)^3(x-\xi_1)^4+4(x-\xi_0)^4(x-\xi_1)^3\\\\
x\in(\xi_0,\xi_1)=4(x-\xi_0)^3(x-\xi_1)^3(x-\xi_1+x-\xi_0)\\\\
=4(x-\xi_0)^3(x-\xi_1)^3\bigl[2x-(\xi_1+\xi_0)\bigr]$\\\\
$\begin{cases}\eta'(\xi_0)=0\\
\eta'(\xi_1)=0\end{cases}\hspace{0.5cm}\star$\\\\
La funci\'{o}n $\eta$ satisface $\star,$ por lo tanto \[\int_{x_0}^{x_1}\eta(x)\phi(x)dx=\int_{\xi_0}^{\xi_1}\eta(x)\phi(x)dx>0\hspace{0.5cm}(\rightarrow\leftarrow)\]\\\\
Luego $\phi=0.$\\\\\\
\end{proof}
Sabido es por todos, que la originalidad y productividad del Suizo, le han otorgado un lugar en la Historia de las Matem\'{a}ticas y en la ciencia en general; la cual ser\'{i}a una utop\'{i}a de no ser gracias a su exagerada Brillantez.\\\\
Euler, hoy 300 a\~{n}os despu\'{e}s de su Natalicio, vive en todos los rincones de las Matem\'{a}ticas y F\'{i}sica; inclusive en la moderna teor\'{i}a de las $\textit{Superstrings}$ le encontramos. Tal teor\'{i}a ha estado evolucionando hacia atr\'{a}s desde su descubrimiento accidental en 1968. Esta es la raz\'{o}n de que la teor\'{i}a parezca tan extra\~{n}a y poco familiar para la mayor\'{i}a de los f\'{i}sicos.\\\\
Ella naci\'{o} casi por casualidad en 1968 cuando dos j\'{o}venes f\'{i}sicos te\'{o}ricos, Gabriel Veneziano y Mahiko Suzuki, estaban hojeando independientemente libros de matem\'{a}ticas, buscando funciones matem\'{a}ticas que describieran las interacciones de part\'{i}culas fuertemente interactivas. Mientras estudiaban en el CERN, tropezaron independientemente con la funci\'{o}n beta de Euler, una funci\'{o}n desarrollada en el siglo XVIII por el Suizo. Se quedaron sorprendidos al descubrir que la funci\'{o}n beta ajustaba casi todas las propiedades requeridas para describir las interacciones fuertes de part\'{i}culas elementales. Hoy, esta funci\'{o}n beta se conoce en f\'{i}sica con el nombre de modelo de Veneziano, que ha inspirado varios miles de art\'{i}culos de investigaci\'{o}n, iniciado una escuela importante en f\'{i}sica y ahora tiene la pretensi\'{o}n de unificar todas las leyes f\'{i}sicas.\\\\
Se llama $\textit{beta de Euler}$ (o integral euleriana de primera especie) a la funci\'{o}n $(p.,q)\mapsto B(p,q),$ continua en su campo de definici\'{o}n, que viene dada por la expresi\'{o}n: \begin{gather}B(p,q)=\int_{\rightarrow 0}^{\rightarrow 1}x^{p-1}(1-x)^{q-1}dx,\hspace{0.5cm}\forall p,q\in]0,+\infty[\end{gather}\\
Mediante el cambio de variable $x\mapsto y=1-x,$ se obtiene que $B(p,q)=B(q,p).$\\
Otras expresiones interesantes de $B(p,q)$ son:\\
\[B(p,q)=\int_{\rightarrow 0}^{+\infty}\dfrac{t^{p-1}}{(1+t)^{p+q}}dt=\int_{\rightarrow 0 }^{1}\dfrac{u^{p-1}+u^{q-1}}{(1+u)^{p+q}}du=2\int_{\rightarrow 0}^{\rightarrow \pi/2}(\sin^{2p-1}\theta)(\cos^{2q-1}\theta)d\theta\]\\
De esta \'{u}ltima se obtiene en particular que $B(1/2,1/2)=\pi.$\\\\
Se verifican las siguientes relaciones de recurrencia:\\\\
$B(p,q+1)=\dfrac{q}{p+q}B(p,q)\hspace{0.5cm}\text{y}\hspace{0.5cm}B(p+1,q)=\dfrac{p}{p+q}B(p,q)$\\\\
con ellas, si se conociera $B(p,q)$ para $p,q\in]0,1],$ se conocer\'{i}a $B(p,q)$ para cualesquiera $p,q>0.$ Como $B(1,1)=1,$ resulta que \[B(m,n)=\dfrac{(m-1)!(n-1)!}{(m+n-1)!},\hspace{0.5cm}\forall m,n\in\mathbb{N}\]\\\\
Las funciones $\Gamma$ y $B$ est\'{a}n ligadas mediante la relaci\'{o}n $$B(p,q)=\dfrac{\Gamma(p)\Gamma(q)}{\Gamma(p+q)}$$\\\\
De esta f\'{o}rmula se desprende que $\Gamma(1/2)=\sqrt{\pi}$ y que $\int_0^{+\infty}e^{-x^2}dx=1/2\sqrt{\pi}.$\\\\
Finalizamos estas notas, con un aforismo Euleriano:\\
\fontfamily{pzc}\selectfont\\
\begin{quote}"Ya que la f\'{a}brica del universo es m\'{a}s que perfecta y es el trabajo de un Creador m\'{a}s que sabio, nada en el universo sucede en el que alguna regla de m\'{a}ximo o m\'{i}nimo no aparezca''.
\end{quote}

\end{document}